\documentclass[12pt]{article}
\usepackage[utf8]{inputenc}

\usepackage{setspace}                   % espaçamento flexível

\usepackage{amsmath}
\usepackage{amscd}
\usepackage{amsthm}
\usepackage{amsxtra}
\usepackage{amssymb}
\usepackage{mathrsfs}
\usepackage{graphics}
\usepackage{tikz}
\usetikzlibrary{matrix}
\usepackage[all,cmtip,2cell]{xy}
\UseAllTwocells

\usepackage{enumitem}

\usepackage[all,cmtip]{xy}

\theoremstyle{plain}

\newtheorem{theorem}[subsection]{Theorem}
\newtheorem{proposition}[subsection]{Proposition}
\newtheorem{lemma}[subsection]{Lemma}
\newtheorem{corollary}[subsection]{Corollary}
\newtheorem{definition}[subsection]{Definition}

\theoremstyle{definition}

\newtheorem{aforisma}[subsection]{}

\newtheorem{remark}[subsection]{Remark}

\usepackage{verbatim}

\immediate\write18{texcount -tex -sum  \jobname.tex > \jobname.wordcount.tex}

\providecommand{\keywords}[1]
{
  \small	
  \textbf{\textit{Keywords--}} #1
}

\title{Elementary schematic geometry}
\author{Thiago Alexandre \\
        \small University of S\~{a}o Paulo
}

\begin{document}

\maketitle

\begin{abstract}
  This paper exposes the language of geometric contexts and elementary schemes, which is a functorial formalism to study categories of geometric objects such as schemes, topological manifolds, differential manifolds, analytic manifolds, etc. Through the theory of Grothendieck topologies and geometric contexts, geometry turns out to be the study of local properties which are stable under base change.

\end{abstract} \hspace{10pt}


\section*{Introduction}

 The term scheme was coined the first time by Chavalley, who defined a scheme in [4] as a certain system of local commutative rings associated to the subvarietes of an integral algebraic variety. On the other hand, inspired by the works of \'{E}lie Cartan, Serre defined an algebraic variety in [3] as an instance of ringed spaces, i.e., pairs $(X,\mathscr{A})$ where $X$ is a topological space and $\mathscr{A}$ is a sheaf of commutative rings on the open subsets of $X$. At the time, both languages had they restrictions: Chevalley defined schemes only for integral algebraic varieties, which excludes important cases in algebraic geometry, as algebraic varieties over non-reduced rings, while Serre only defined algebraic varieties over algebraically closed fields. Then, Grothendieck made a Copernican Turn when he introduced the language of schemes as exposed in [2], which solves the two previous restrictions, through a functorial formalism and the use of nilpotents, which later would become a crucial aspect in Grothendieck's algebraic geometry, specially for the definition of \'{e}tale and smooth morphisms of schemes, with recent reverberations even in synthetic differential geometry. 
 
 Although the functorial formalism in algebraic geometry was already present in [2], a strict presentation of the language of schemes only through categorical terms, without any mention to ringed spaces, appeared at the first time in [5]. For a brief introduction to this strictly functorial algebraic geometry, we also indicate [6]. Then, it became clear to algebraic geometers that the notion of schemes is strictly functorial, and that modelizing the category of schemes in locally ringed spaces was only an intermediary step.
 
 It happens that this new functorial geometry of schemes is not restricted only to algebraic geometry, but can also be employed, as we shall see in this paper, to other \emph{geometric contexts}, as topological geometry, differential geometry, analytic geometry, etc. In order to give a brief illustration of this, consider, for example, the category $C$ formed by the open subsets of cartesian spaces $\mathbb{R}^{n}$, where $n$ varies over the integers $\geq 0$. A topological manifold can be defined as a topological space $X$ which is locally homeomorphic to an object of $C$ \footnote{There are other technical topological conditions in the definition of topological manifolds, which are certain separability prescriptions, but they are not essential for our reasoning here, and there is no need to suppose such conditions for our purposes. }. Hence, every topological manifold $X$ admits an open covering $(U_{i} \rightarrow X)_{i \in I}$ such that $U_{i}$ is isomorphic to an open subset of $\mathbb{R}^{n_{i}}$ for some integer $n_{i} \geq 0$. Now, from the previous definition, we can verify that the evident functor
          $$ \mathcal{M} \longrightarrow \widehat{C}, \quad X \mapsto (U \mapsto Hom_{\mathcal{M}}(U,X)) $$
is faithful fully, where $\mathcal{M}$ denotes the full subcategory of topological spaces formed by the topological manifolds and $\widehat{C}$ denotes the category of presheaves over $C$. Moreover, it is possible to define precisely (only in categorical language) when a presheaf $F$ over $C$ is representable by a topological manifold. First, we define the category $\mathcal{E}$ of sheaves over the site $(C,J)$, where $J$ is the Grothendieck topology over $C$ induced from the open coverings. Then, we verify that each topological manifold is an object of $\mathcal{E}$ through the functor $\mathcal{M} \rightarrow \widehat{C}$ above. Now, we can extend the class of open immersions of $C$ to $\mathcal{E}$ (formulating a diagramatic condition). Finally, we can verify that a presheaf $F$ over $C$ is representable by a topological manifold if, and only if, it is an object of $\mathcal{E}$ and there exists a family of open immersions $(U_{i} \rightarrow F)_{i \in I}$, where each $U_{i}$ is representable by an object of $C$, inducing a canonical epimorphism
   $$ \coprod_{i \in I} U_{i} \longrightarrow F $$
in the category $\mathcal{E}$. A family of open immersions of the previous type for a sheaf $F$ is called an open atlas of $F$. Hence, we can redefine a topological manifold as a sheaf over $(C,J)$ which admits at least one open atlas. This formalises the idea that a topological manifold is the gluing of coherent family of euclidean open sets.

The same argument given in the above paragraph can be replied for differential manifolds (resp. (complex) analytic manifolds), regarded we start from the subcategory $\mathbf{C}_{\infty}$ of $C$ with the same objects, but admitting only $\mathcal{C}^{\infty}$-functions as morphisms, and considering the category $Diff$ of differential manifolds in place of $\mathcal{M}$.
 
 Several guidances to formalise these \emph{schematic categories} are known today ([11], [12], [13], [14], and [14] \footnote{We also indicate the reader to see [10].}). In this paper, we follow To\"{e}n-Vaqui\'{e}-Vezzosi's orientation, which concerns about geometric contexts, i.e., triples $(C,J,\mathbf{P})$ where $C$ is a small category (the category of local models), $J$ is a Grothendieck topology over $C$ and $\mathbf{P}$ is an admissible class of arrows in $C$ (the class of local morphisms, codifying an abstract notion of open immersion). Yet, an axiomatic for these triples can be very flexible, and we suppose less conditions than the ones defined in [11]. 
 
 We can define the notion of scheme in every geometric context $(C,J,\mathbf{P})$ as certain full subcategory of the category of sheaves over the site $(C,J)$ (in particular, these elementary schemes will be functors from $C^{o}$ to the category of sets $Ens$, where $C^{o}$ denotes the dual category of $C$). In the language of geometric contexts, geometry is codified as \emph{local properties which are invariant under base change} \footnote{In this sense, the language of elementary schemes is a synthesis between Riemann and Klein conceptions of geometry: the locality of the geometric properties is an inspiration of Riemann, while the invariance under base change is an inspiration of Klein.}.
 
 The language of geometric contexts and elementary schemes has several deep consequences which lies beyond the simple generalisation of the category of schemes and the uniformity of geometry in mathematics. One of the major goals is the study of $\mathbb{F}_{1}$-algebraic geometry (and even possible new geometries) in this setting, which has deep relations with Riemann Hypothesis. 
 
 After introducing geometric contexts, the three next steps (which we do not give in this paper) would be (1). A (co)homological study of elementary schemes in geometric contexts; (2). The formalisation, in the language of geometric contexts (eventually under suitable hypothesis), of some crucial geometric notions as (co)tangent bundle, normal bundle, tangent cone, etc., as well eventual relations that these formal geometric constructions would have with (co)homologies theories over elementary schemes; (3). Morphisms of geometric contexts and transfer of geometric data between schematic categories.

 All the proofs in this paper are elementary, in the sense that they follow formally from the definitions without any artificial or technical enhancement.

\subsection*{General conventions} 
  We assume that reader has familiarity with basic categorical language, including Grothendieck topologies, sites, categories of sheaves over sites, etc. We indicate [1] as a main reference on this subject, but the preliminary chapter of [16] is also a nice survey on topos theory. If the reader has preferences for an English textbook, we strongly recommend [7] or [8], and the preliminary chapter of [9] is yet another possibility. However, a brief exposition about categories of sheaves is given in the begin of the first section. Concerning technical general conventions, we assume two Grothendieck universes $\mathsf{U}$ and $\mathsf{V}$ such that $\omega \in \mathsf{U} \in \mathsf{V}$ (see \textsc{Expos\'{e} I} of [1] as a main reference about Grothendieck universes). The term set will be reserved for the elements of $\mathsf{U}$, while the term class will be used in latus sensus, indicating both an element of $\mathsf{U}$ as one of $\mathsf{V}$. If $X$ and $Y$ are two objects in a category $C$, then we denote by $Hom_{C}(X,Y)$ the set of morphisms from $X$ to $Y$ in $C$. The symbols $Ob(C)$ and $Fl(C)$ denote respectively the classes of objects and arrows of a given category $C$. Given an object $U$ in a category $C$, we indicate by $Id_{U}$ the identity arrow of $U$. Given two categories $C$ and $D$, we denote by $\underline{Hom}(C,D)$ the category of functors from $C$ to $D$, with the morphisms being natural transformations. When $C$ and $D$ are co-complete, we denote by $\underline{Hom}_{!}(C,D)$ the full subcategory of $\underline{Hom}(C,D)$ formed by the functors which commute with colimits. A category $C$ will be called small if $Ob(C)$ and $Fl(C)$ are isomorphic to an element of $\mathsf{U}$, and locally small, if for any objects $X,Y \in Ob(C)$, the set $Hom_{C}(X,Y)$ is isomorphic to an element of $\mathsf{U}$. We denote by $Ens$ the category of small sets, i.e., the category of sets which belong to the universe $\mathsf{U}$. The term topos will always make reference to a Grothendieck topos. Given a category $C$ (resp. a site $(C,J)$), we denote always by $\widehat{C}$ (resp. $Sh(C,J)$) the category of presehaves (resp. sheaves) over $C$ (resp. over $(C,J)$).  If $\mathbf{P}$ is a class of arrows in a category $C$, we say that an arrow $\varphi$ of $C$ is a $\mathbf{P}$-morphism (resp. $\mathbf{P}$-monomorphism, $\mathbf{P}$-epimorphism, $\mathbf{P}$-isomorphism) if $\varphi$ is an element of $\mathbf{P}$ (resp. a monomorphism in $\mathbf{P}$, an epimorphism in $\mathbf{P}$, an isomorphism in $\mathbf{P}$).

\section{Geometric contexts}

We start fixing some terminology and recalling some well known results about categories of (pre)sheaves. \\

  \emph{Presheaves} - Let $C$ be a small category. The category of presheaves over $C$, denoted by $\widehat{C}$, is the category of functors from $C^{o}$ to the category of sets $Ens$. The Yoneda functor
    $$ h: C \longrightarrow \widehat{C}, \quad U \mapsto (V \mapsto Hom_{C}(V,U)) $$  
satisfies the following condition (known as Yoneda's Lemma): given any presheaf $X$ over $C$, the function
    $$ Hom_{\widehat{C}}(h_{U},X) \longrightarrow X(U), \quad f \mapsto f_{U}(Id_{U}) $$
is a natural bijection (in regard to the variables $U$ and $X$). As a consequence, we have that the Yoneda functor $h$ is also faithful fully. In virtue of this property, we often denote just by $U$ the representable presheaf $h_{U}$ associated to an object $U$ of $C$. Moreover, $\widehat{C}$ is complete and co-complete, and characterized by the following universal property: if $\mathcal{E}$ is any co-complete category, the functor
       $$ \underline{Hom}_{!}(\widehat{C}, \mathcal{E}) \longrightarrow \underline{Hom}(C, \mathcal{E}), \quad F \mapsto F \circ h $$
is an equivalence of categories. Given a presheaf $F$ over $C$ and $U \in Ob(C)$, we call the elements of $F(U)$ sections of $F$ at $U$. If $V \rightarrow U$ is a morphism of $C$ and $s \in F(U)$, we also denote by $s \vert_{V}$ the image of $s$ in $F(V)$ through the function $F(U) \rightarrow F(V)$ induced by the functor $F$. \\  
  
  \emph{Cartesian squares} - A diagram of the form
  \[
  \xymatrix{
    X' \ar[r] \ar[d]   &    Y' \ar[d] \\
    X \ar[r]           &    Y
  }
  \]
in a category $C$ is called a cartesian square (\emph{pullback}, \emph{fibred product}) if $X'$ represents the limit of the diagram
\[
\xymatrix{
              &   Y' \ar[d] \\
  X \ar[r]    &   Y.
}
\]
 An arrow $X \rightarrow Y$ in $C$ is called \emph{cartesian} if any morphism of the form $Y' \rightarrow Y$ in $C$ can be completed to a cartesian square of the form
\[
\xymatrix{
  X' \ar[r] \ar[d]   &   Y' \ar[d] \\
  X \ar[r]           &   Y.
}
\]
The notions of co-cartesian squares and co-cartesian morphisms are defined dually. \\

 \emph{Admissible class of arrows} -  Let $C$ be a category and $\mathbf{P}$ be a class of arrows in $C$. We say that $\mathbf{P}$ is admissible if it satisfies the following conditions:
  \begin{enumerate}
     \item (\emph{Contains all identities}) For every $U \in Ob(C)$, the identity arrow $Id_{U}: U \rightarrow U$ is a $\mathbf{P}$-morphism.
     \item (\emph{Stability under base change}) For every Cartesian square
     \[
     \xymatrix{
       U' \ar[r]^{\varphi'} \ar[d]   &   V' \ar[d]\\
       U \ar[r]_{\varphi}            &   V
     }
     \]
 if $\varphi$ is a $\mathbf{P}$-morhpism, then $\varphi'$ is also a $\mathbf{P}$-morphism.    
     \item (\emph{Stability under composition}) If $\varphi: U \rightarrow V$ and $\psi: V \rightarrow W$ are two composable arrows in $C$, where $\varphi$ and $\psi$ are both $\mathbf{P}$-morphisms, then $\psi \circ \varphi$ is also a $\mathbf{P}$-morphism. 
     
  \end{enumerate}
We can verify easily that if $\mathbf{P}$ is an admissible class of arrows in $C$, then $\mathbf{P}$ contains all the \emph{isomorphisms} in $C$. \\ 

\emph{Sieve} - Let $C$ be a small category. A sieve of an object $U \in Ob(C)$ is a subfunctor of the representable functor $h_{U}$. If $\varphi: U \rightarrow V$ is a morphism of $C$ and $R$ is a sieve of $V$, then we denote by $\varphi^{-1}(R)$ the sieve of $U$ defined by the Cartesian square:
\[
\xymatrix{
 \varphi^{-1}(R) \ar[d] \ar[r]    &  R \ar[d] \\
 h_{U} \ar[r]_{h_{\varphi}}               &  h_{V}.
}
\]
 \\

\emph{Grothendieck topology} - We recall that a Grothendieck topology $J$ on $C$ is a function which assigns to each object $U$ of $C$ a set $J(U)$ of sieves of $U$ such that 
 \begin{enumerate}
   \item If $\varphi: U \rightarrow V$ is a morphism in $C$ and $R \in J(V)$, then $\varphi^{-1}(R) \in J(U)$
   \item If $R$ and $S$ are two sieves of an object $U \in Ob(C)$, where $R \in J(U)$ and, for every morphism $\varphi: V \rightarrow U$ in $R$, we have $\varphi^{-1}(S) \in J(V)$, then $S \in J(U)$.
   \item For every $U \in Ob(C)$, $h_{U} \in J(U)$.
   
 \end{enumerate} 
A sieve $R$ of $U$ is generated by a family of arrows of the form $(\rho_{i}:U_{i} \rightarrow U)_{i \in I}$ if each arrow $\varphi:V \rightarrow U$ in $R$ factors through an arrow $\rho_{i}:U_{i} \rightarrow U$ for some $i \in I$. Conversely, a family of arrows of the form $(\rho_{i}:U_{i} \rightarrow U)_{i \in I}$ is called a $J$-covering if the sieve generated by this family is an element of $J(U)$. \\

\emph{Grothendieck pre-topology} - A Grothendieck pre-topology $Cov$ over a small category $C$ is a function assigning to each object $U \in Ob(C)$ a set $Cov(U)$ of families of arrows of the form $(U_{i} \rightarrow U)_{i \in I}$, called coverings, such that the following conditions are verified
 \begin{enumerate}
   \item If $(U_{i} \rightarrow U)_{i \in I}$ is a covering, then, for every $i \in I$, the arrow $U_{i} \rightarrow U$ is cartesian.
   \item If $\varphi: V \rightarrow U$ is a morphism of $C$ and $(U_{i} \rightarrow U)_{i \in I}$ is a covering, then, $(V \times_{U} U_{i} \rightarrow V)_{i \in I}$ is a covering.
   \item If $(U_{i} \rightarrow U)_{i \in I}$ is a covering and, for every $i \in I$, $(V_{i,a} \rightarrow U_{i})_{a \in A_{i}}$ is a covering, then, the family $(V_{i,a} \rightarrow \rightarrow U_{i} \rightarrow U)$, given by composition, is a covering.
   \item If $\varphi: V \rightarrow U$ is any isomorphism in $C$, then the singleton $\{\varphi: V \rightarrow U\}$ is a covering.
   
 \end{enumerate}
If $Cov$ is a Grothendieck pre-topology on $C$, then, we can always form a topology $J$, defining, for each $U \in Ob(C)$, the set $J(U)$ of sieves $R$ of $U$ which contains a sieve $R'$ generated by a covering $(U_{i} \rightarrow U)_{i  \in I}$ in $Cov(U)$. In this case, we say that $J$ is the Grothendieck topology generated by $Cov$. \\

\emph{Sites} - A site is a pair $(C,J)$ where $C$ is a \emph{small} category and $J$ is a Grothendieck topology in $C$.\\

\emph{Sheaves} - Let $(C,J)$ be a site. A presheaf $F$ over $C$ is called a sheaf over the site $(C,J)$ if for every object $U$ of $C$ and every $J$-sieve $R$, the function
     $$ Hom_{\widehat{C}}(h_{U}, F) \longrightarrow Hom_{\widehat{C}}(R,F), \quad f \mapsto f \vert_{R} = f \circ \imath_{R},$$
where $\imath_{R}: R \hookrightarrow h_{U}$ denotes the inclusion arrow, is a bijection. We denote by $Sh(C,J)$ the category of sheaves over $(C,J)$, which is the full subcategory of $\widehat{C}$ formed by the sheaves. When the Grothendieck topology $J$ is generated by a Grothentieck pre-topology $Cov$, a presheaf $F$ is a sheaf if, and only if, for every $U \in Ob(C)$ and every $J$-covering $(\rho_{i}: U_{i} \rightarrow U)_{i \in I}$, the following condition is verified: if $(s_{i})_{i \in I}$ is family of sections of $F$, with $s_{i} \in F(U_{i})$ for every $i \in I$, such that $s_{i} \vert_{U_{ij}} = s_{j} \vert_{U_{ij}}$ for all $i,j \in I$ (where $U_{ij}$ always denote the fibred product $U_{i} \times_{U} U_{j}$ in this situation), then, there exists a unique section $s \in F(U)$ such that $s \vert_{U_{i}} = s_{i}$ for all $i \in I$. The inclusion functor $Sh(C,J) \hookrightarrow \widehat{C}$ always admits a left adjoint $a: \widehat{C} \rightarrow Sh(C,J)$, which associates to each presheaf $F$ over $C$, it's associated sheaf $a(F)$. Moreover, the functor $a$ is left exact, i.e., $a$ commutes with finite limits. We can verify that a morphism $f: F \rightarrow G$ of sheaves is an epimorphism precisely when $f$ is local surjective, i.e., when for every $U \in Ob(C)$ and $s \in G(U)$, there exists a $J$-covering $(U_{i} \rightarrow U)_{i \in I}$ such that $s \vert_{U_{i}}$ is the image of some section of $F(U_{i})$ by the arrow $f_{U_{i}}: F(U_{i}) \rightarrow G(U_{i})$. Using the previous characterization of sheaf epimorphisms, we can verify that in any category of sheaves, the \emph{epimorphisms are stable by base change}. The topology $J$ is called \emph{sub-canonical} if for every $U \in Ob(C)$, the representable presheaf $h_{U}$ is a sheaf over $(C,J)$. If $(C,J)$ a site, $J$ is sub-canonical, and $(U_{i} \rightarrow U)_{i \in I}$ is a family of morphisms in $C$, then $(U_{i} \rightarrow U)_{i \in I}$ is a $J$-covering precisely when the induced arrow 
     $$ \coprod_{i \in I} h_{U_{i}} \longrightarrow h_{U} $$
is a \emph{sheaf epimorphism}, which is equivalent to say that the above arrow, considered as a morphism of \emph{presheaves}, is a local epimorphism (see \emph{Definition 16.1.5} of [7].). A morphism of sheaves $f: F \rightarrow G$ is a monomorphism precisely when $f_{U}: F(U) \rightarrow G(U)$ is injective for every $U \in Ob(C)$. \\

\emph{Topos} - A locally small category $\mathcal{E}$ is called a topos if it is equivalent to the category of sheaves over some site, which is equivalent to say that $\mathcal{E}$ admits finite limits and it is a left exact reflective localisation of a presheaf category, i.e., there exists a small category $C$ and a left exact functor $a: \widehat{C} \rightarrow \mathcal{E}$ which admits a faithful fully right adjoint $\gamma: \mathcal{E} \rightarrow \widehat{C}$. The functor $a: \widehat{C} \rightarrow \mathcal{E}$ is called the sheafification functor. \\

\textsc{Convention} - Through all this work, every Grothendieck topology is supposed to be generated by a Grothendieck pre-topology \footnote{This is the case in all the examples of sites coming from geometry. Indeed, the reader can even exchange the term `Grothendieck topology' by `Grothendieck pre-topology' in the sequel, since we always work only with coverings. We remark the many authors uses the terminology `Grothendieck topology' to define Grothendieck pre-topologies. We could also have opted this terminology here. Yet, this has future technical disadvantages and we prefer to maintain the definitions as presented in [1]. However, this is not a genuine restriction, because any topos $\mathcal{E}$ admits a \emph{standard site} $(C,J)$, which is a site where $C$ admits finite limits, and $J$ is generated by a pre-topology. }.

\begin{definition}\label{Geometric context}
  A \emph{geometric context} consists of a triple $(C,J,\mathbf{P})$ where $C$ is a small category, $J$ is a Grothendieck topology on $C$ and $\mathbf{P}$ is a class of arrows in $C$, such that the following conditions are verified \\

  (GC1). $C$ admits \emph{finite products} \\
  
  (GC2). $J$ is \emph{sub-canonical} \footnote{We should be able to define a more general notion of geometric context which does not suppose this condition, since it excludes important cases of Grothendieck topologies, as the Voevodsky's $h$-topology in algebraic geometry, which is not subcanonical. Moreover, this condition does not seem to be essential from a strict conceptual perspective. Yet, we follow To\"{e}n-Vaqui\'{e} here and assume (GC2), since it also facilitates many future arguments. Actually, it would be more appropriated, at least theoretically, to define geometric contexts directly in a topos $\mathcal{E}$, without necessarily mention a site of definition of $\mathcal{E}$. } \\ 
  
  (GC3). $\mathbf{P}$ is \emph{admissible} \\
  
  (GC4). $\mathbf{P}$ is $J$-\emph{local}: If $\varphi: U \rightarrow V$ is an arrow of $C$ and there exists a $J$-covering $(\rho_{i}: U_{i} \rightarrow U)_{i \in I}$ such that $\rho_{i} \in \mathbf{P}$ and $\varphi \circ \rho_{i} \in \mathbf{P}$ for every $i \in I$, then $\varphi \in \mathbf{P}$. \\
  
  (GC5). $J$ is $\mathbf{P}$-\emph{generated}: For every object $U$ of $C$ and every $J$-covering $(\rho_{a}: U_{a} \rightarrow U)_{a \in A}$, there exists a family of morphisms $(\varphi_{i}: V_{i} \rightarrow U)_{i \in I}$, with $\varphi_{i} \in \mathbf{P}$ for every $i \in I$, such that each arrow $\rho_{a}: U_{a} \rightarrow U$ factors through $\varphi_{i}: V_{i} \rightarrow U$ for some $i \in I$:
  \[
  \xymatrix{
    U_{a} \ar[r]^{\rho_{a}} \ar[d]   &     U \\
    V_{i} \ar[ur]_{\varphi_{i}}
  }
  \]
In other words, every covering admits a $\mathbf{P}$-refinement. \\
  
  (GC6). $\mathbf{P}$ is \emph{locally cartesian}: If $\varphi: U \rightarrow V$ is a $\mathbf{P}$-morphism, then there exists a $J$-covering $(\rho_{i}: U_{i} \rightarrow U)_{i \in I}$, with $\rho_{i} \in \mathbf{P}$ for every $i \in I$, such that the morphisms $\varphi \circ \rho_{i}$, for $i \in I$, are all \emph{cartesian}.
  
\end{definition}

\emph{Terminology} - Given a geometric context $(C,J,\mathbf{P})$, we call local models the objects of $C$ and open immersions (resp. local arrows) the $\mathbf{P}$-monomorphisms (resp. $\mathbf{P}$-morphisms).\\

We recall that any morphism $f: F \rightarrow G$ in a category of sheaves factors through $F \xrightarrow{p} Im(f) \xrightarrow{i} G$, where $i$ is a monomorphism and $p$ is an epimorphism. Moreover, this factorization is unique up to isomorphism, and if $f$ is a monomorphism (resp.  $f$ is an epimorphism), then $p$ (resp. $i$) is an isomorphism. To construct the sheaf $Im(f)$, form the \emph{co-cartesian} square
\[
\xymatrix{
  F \ar[r]^{f} \ar[d]_{f}    &   G \ar[d]^{q} \\
  G \ar[r]_{r}                   &   H
}
\]
and then, define $i:Im(f) \rightarrow G$ as the equaliser of the parallel arrows $q,r: G \rightarrow H$, and $p: F \rightarrow Im(f)$ as the restriction of $f$.

\begin{definition}\label{Sheaf open immersion}
  Let $(C,J,\mathbf{P})$ be a geometric context. A morphism $f: F \rightarrow G$ of sheaves over $(C,J)$ is called an open immersion if it satisfies the following condition: for every morphism $h_{U} \rightarrow G$, with $U \in Ob(C)$, there exists a family of $\mathbf{P}$-morphisms $(U_{i} \rightarrow U)_{i \in I}$ in $C$, such that $f': F \times_{G} h_{U} \rightarrow h_{U}$ is a \emph{monomorphism} and the image of the canonical arrow
        $$ \coprod_{i \in I} h_{U_{i}} \longrightarrow h_{U} $$
is isomorphic to the image of the canonical arrow 
         $$ F \times_{G} h_{U} \longrightarrow h_{U}. $$

\end{definition}

\begin{aforisma}\label{Open immersions are mono}
 It follows from the definition (\ref{Sheaf open immersion}) that every open immersion of sheaves in a geometric context is a \emph{monomorphism}. Indeed, let $f: F \rightarrow G$ be a open immersion of sheaves in a geometric context $(C,J,\mathbf{P})$. In order to proof that $f$ is a monomorphism, it is enough to verify that given two arrows of the form $s,t: h_{U} \rightarrow F$, where $U \in Ob(C)$ and $f \circ s = f \circ t$, we have $s=t$ (by Yoneda's lemma). Let $g= f \circ s=f \circ t$ be the composed arrow from $h_{U}$ to $G$ and consider the cartesian square
 \[
 \xymatrix{
     F \times_{G} h_{U} \ar[r]^{f'} \ar[d]_{g'}  &    h_{U} \ar[d]^{g} \\
    F \ar[r]_{f}                                 &    G.
 }
 \]
Since $f$ is supposed to be an open immersion, $f'$ is a monomorphism. Moreover, it follows from the universal property of fibred products, that there exists a unique arrow $s': h_{U} \rightarrow F \times_{G} h_{U}$ (resp. $t': h_{U} \rightarrow F \times_{G} h_{U}$) such that the diagram
\[
\xymatrix{
  h_{U} \ar[dr]^{s'}  \ar@/_1pc/[ddr]_{s} \ar@/^1pc/[drr]^{Id}   &                                               &        \\
                                                   &   F \times_{G} h_{U} \ar[d]_{g'} \ar[r]^{f'}  &  h_{U} \ar[d]^{g} \\
                                                   &   F \ar[r]_{f}                                &   G    
}
\]
(resp. the diagram
\[
\xymatrix{
  h_{U} \ar[dr]^{t'}  \ar@/_1pc/[ddr]_{t} \ar@/^1pc/[drr]^{Id}   &                                               &        \\
                                                   &   F \times_{G} h_{U} \ar[d]_{g'} \ar[r]^{f'}  &  h_{U} \ar[d]^{g} \\
                                                   &   F \ar[r]_{f}                                &   G    
}
\]
commutes. Hence, $f' \circ s' = Id = f' \circ t'$, which implies $s' =t'$ (since $f$ is a monomorphism). Therefore, $s=g' \circ s' = g' \circ t' = t$.

\end{aforisma}

\begin{proposition}\label{Stability properties of open immersions}
   For every geometric context $(C,J,\mathbf{P})$ the class $Ouv$ of open immersions in $Sh(C,J)$ induced from $\mathbf{P}$ satisfies the following properties
   \begin{enumerate}
      \item $Ouv$ contains all identities
      \item $Ouv$ is stable by base change
      \item $Ouv$ is local: If $f: F \rightarrow G$ is a morphism of sheaves such that for each arrow $h_{U} \rightarrow G$, with $U \in Ob(C)$, the canonical morphism $F \times_{G} h_{U} \rightarrow h_{U}$ is an open immersion, then $f: F \rightarrow G$ is an open immersion
      \item $Ouv$ is stable by composition
      \item If $f: F \rightarrow G$ and $g: G \rightarrow H$ are two composable arrows of sheaves such that $g \circ f$ and $g$ are open immersions, then $f$ is also an open immersion.
      
   \end{enumerate}

\end{proposition}

\begin{proof}
  The assertions (1), (2) and (3) are easy formal consequences of the definitions and the assertions (4) and (5) follow immediately from (2) and (3).

\end{proof}

\begin{lemma}
   Let $f: F \rightarrow h_{U}$ be a morphism of sheaves in a geometric context $(C,J,\mathbf{P})$. Then, $f$ is an open immersion of sheaves if, and only if, $f$ is a monomorphism and there exists a family of $\mathbf{P}$-morphisms $(U_{i} \rightarrow U)_{i \in I}$ in $C$, such that the image of $f$ is isomorphic to the image of the canonical induced morphism
           $$ \coprod_{i \in I} h_{U_{i}} \longrightarrow h_{U}. $$

\end{lemma}

\begin{proof}
  It $f$ is an open immersion, then the reciprocal in the lemma follows easily from the definition (\ref{Sheaf open immersion}) and from the fact that open immersions are stable by base change ((2) of (\ref{Stability properties of open immersions})). Now, suppose that $f$ is a monomorphism and there exists a family of $\mathbf{P}$-morphisms $(U_{i} \rightarrow U)_{i \in I}$ in $C$, such that the image of $f$ coincides (isomorphically) to the image of the induced arrow of sheaves
         $$ \coprod_{i \in I} h_{U_{i}} \longrightarrow h_{U}. $$
Moreover, let $h_{V} \rightarrow h_{U}$ be a morphism of sheaves for some $V \in Ob(C)$, corresponding, then, to a morphism $V \rightarrow U$ in $C$. Hence, by the fact that $\mathbf{P}$ is \emph{locally cartesian} (condition (GC6) of (\ref{Geometric context})), there exists, for each $i \in I$, a $J$-covering $(\rho_{i,a}:U_{i,a} \rightarrow U_{i})_{a \in A_{i}}$  with $\rho_{i,a} \in \mathbf{P}$ for every $a \in A_{i}$, such that the composed arrows $U_{i,a} \rightarrow U_{i} \rightarrow U$, which are all $\mathbf{P}$-morphisms, are also all \emph{cartesian}, which implies in the existence of cartesian squares of the form
\[
\xymatrix{
  V_{i,a} \ar[d] \ar[r]  &   V \ar[d] \\
  U_{i,a}  \ar[r]                     &   U,
}
\]
in $C$, and since $\mathbf{P}$ is \emph{stable by base change}, each one of the arrows $V_{i,a} \rightarrow V$ is a $\mathbf{P}$-morphism. If we consider the set given by the disjoint union
     $$ A =_{df} \bigsqcup_{i \in I} A_{i}, $$
then the elements of the set $A$ are pairs $(i,a)$ with $i \in I$ and $a \in A_{i}$, and we have a family of $\mathbf{P}$-morphisms $(V_{\alpha} \rightarrow V)_{\alpha \in A}$.
 Now, with the previous notations, we can verify easily that the canonical arrow $F \times_{h_{U}} h_{V} \rightarrow h_{V}$ is a monomorphism (since monomorphisms are stable by base change) and it's image is isomorphic to the image of the induced morphism 
       $$ \coprod_{\alpha \in A} h_{V_{\alpha}} \longrightarrow h_{V}, $$
 from where we conclude that $f: F \rightarrow h_{U}$ is actually an open immersion in the sense of (\ref{Sheaf open immersion}). 

\end{proof}

\begin{proposition}\label{Open immersion equivalence}
  Let $(C,J,\mathbf{P})$ be a geometric context. A morphism $\varphi: U' \rightarrow U$ in $C$ is a $\mathbf{P}$-monomorphism if, and only if, the morphism of sheaves $h_{\varphi}: h_{U'} \rightarrow h_{U}$ is an open immersion in the sense of (\ref{Sheaf open immersion}).
  
\end{proposition}

\begin{proof}
  To prove the proposition, we use the previous lemma. If $\varphi: U' \rightarrow U$ is a $\mathbf{P}$-monomorphism, then the singleton $\{\varphi: U' \rightarrow U\}$ is itself a family of $\mathbf{P}$-morphisms satisfying the conditions of the previous lemma for the morphism $h_{\varphi}: h_{U'} \rightarrow h_{U}$, from where we conclude that $h_{\varphi}$ is an open immersion, since the Yoneda functor preserves monomorphisms. Now, suppose that $h_{\varphi}: h_{U'} \rightarrow h_{U}$ is an open immersion. Then, again by the previous lemma, $\varphi$ is a monomorphism (because $h_{\varphi}$ is a monomorphism and the Yoneda functor is conservative with regard to monomorphisms \footnote{This means that an arrow $\varphi$ of $C$ is a monomorphism if, and only if, $h_{\varphi}$ is a monomorphism in $\widehat{C}$.}), and  there exists a family of $\mathbf{P}$-morphisms $(\varphi_{i}: U_{i} \rightarrow U)_{i \in I}$ such that the image $Im_{\varphi}$ of $h_{\varphi}$ is isomorphic to the image $Im_{p}$ of the canonical arrow 
      $$ p: \coprod_{i \in I} h_{U_{i}} \longrightarrow h_{U}. $$
From the fact that $h_{\varphi}$ is a monomorphism, it follows that $h_{U'}$ is isomorphic to $Im_{\varphi}$, from where we deduce, for each $i \in I$, an arrow $h_{U_{i}} \rightarrow h_{U'}$, which corresponds to an arrow $\rho_{i}: U_{i} \rightarrow U'$ of $C$ such that the triangle
\[
\xymatrix{
        &   U_{i} \ar[d]^{\varphi_{i}} \ar[dl]_{\rho_{i}} \\
 U' \ar[r]_{\varphi} & U
}
\] 
commutes. With the previous notations, we can verify easily (from the fact that $\mathbf{P}$ is stable by base change), that $\rho_{i}:U_{i} \rightarrow U'$ is a $\mathbf{P}$-morphism. Moreover, the family $(\rho_{i}: U_{i} \rightarrow U')_{i \in I}$ is a $J$-covering such that $\rho_{i} \in \mathbf{P}$ and $\varphi \circ \rho_{i} = \varphi_{i} \in \mathbf{P}$. Therefore, by (GC4) of (\ref{Geometric context}), we conclude that $\varphi$ is a $\mathbf{P}$-monomorphism.

\end{proof}

In the following, we give a list of examples of geometric contexts: \\

\subsection*{Topological geometry context}
 Let $(\mathbf{C}_{0}, ouv, et)$ be the triple where \\
 
 (i) $\mathbf{C}_{0}$ is the small full subcategory of $\mathcal{T}op$ formed by the open subsets of $\mathbb{R}^{n}$, for $n$ varying through the integers $\geq 0$; \\
 
 (ii) $ouv$ is the Grothendieck topology on $\mathbf{C}_{0}$ generated by the open coverings; \\
 
 (iii) $et$ is the class of local homeomorphisms in $\mathbf{C}_{0}$. \\
 
Then, we can verify that the triple $(\mathbf{C}_{0}, ouv, et)$ is a geometric context.

\subsection*{Differential geometry context}
  Let $(\mathbf{C}_{\infty}, O_{\infty}, diff)$ be the triple where \\
  
  (i) $\mathbf{C}_{\infty}$ is the subcategory of $\mathbf{C}_{0}$ with the same objects, but admitting only $\mathcal{C}^{\infty}$-functions as morphisms \\
  
  (ii) $O_{\infty}$ is the Grothendieck topology generated by the open coverings in $\mathbf{C}_{\infty}$, which can be described strictly by differential functions in the following way: an $O_{\infty}$-covering $(\varphi_{i}: U_{i} \rightarrow U)_{i \in I}$ is family of injective $\mathcal{C}^{\infty}$-functions such that, for every $i \in I$ and for every $p \in U_{i}$, the linear transformation $d(\varphi_{i})_{p}$, of the differential function $\varphi_{i}$ at the point $p$, is a bijection, and for for every $x \in U$, there exist $i \in I$ and $p \in U_{i}$ such that $x = \varphi_{i}(p)$. \\
  
  (iii) $diff$ is the class of local diffeomorphisms in $\mathbf{C}_{\infty}$, i.e., arrows $\varphi: U \rightarrow V$ in $\mathbf{C}_{\infty}$ such that, for every $p \in U$, the differential linear transformation $d\varphi_{p}$, of the function $\varphi$ at the point $p$, is a bijection. \\
  
  Then, we can verify that the the triple $(\mathbf{C}_{\infty}, O_{\infty}, diff)$ is a geometric context.

\subsection*{Analytic geometry context}
  Let $(\mathbf{C}_{h}, O_{h}, lis)$ be the triple where \\
  
  (i) $\mathbf{C}_{h}$ is the subcategory of $\mathcal{T}op$ formed by the open subsets of $\mathbb{C}^{n}$, for $n$ varying through the integers $n \geq 0$, with morphisms being holomorphic functions \\
  
  (ii) $O_{h}$ is the Grothendieck topology generated by the open coverings in $\mathbf{C}_{h}$; \\
  
  (iii) $et_{h}$ is the class of locally biholomorphisms. \\
  
The, we can verify that the triple $(\mathbf{C}_{h}, O_{h}, lis)$ is a geometric context.

\subsection*{Algebraic geometry context I}

Let $(Aff, fpqc, fppf)$ be the triple where \\

(i) $Aff$ is the dual category of the category of commutative rings which are of finite presentation \footnote{Actually, the category $Aff$ admits a small skeleton, given by the rings of the form $\mathbb{Z}[T_{1},...,T_{n}]/(P_{1},...,P_{m})$, where $P_{1},...,P_{m}$ are polynomials in $\mathbb{Z}[T_{1},...,T_{n}]$.}. \\

(ii) $fpqc$ is the faithful-flat-quasi-compact Grothendieck topology on $Aff$; \\

(iii) $fppf$ is the class of faithful-flat-finite-presentation morphisms in $Aff$. \\

The, we can verify that $(Aff,fpqc, fppf)$ is a geometric context \footnote{We indicate the references [1] and [5] for the definitions of the topologies $fpqc$ (faithful-flat-quasi-compact) $fppf$ (faithful-flat-finite-presentation), $et$ (\'{e}tale) and $zar$ (Zariski). A definition of the Niesnevich topology $nis$ can be found in [7].}.

\subsection*{Algebraic geometry context II}

Let $(Aff, et, lis)$ be the triple where \\

(i) $Aff$ is the dual category of the category of commutative rings which are of finite presentation. \\

(ii) $et$ is the \'{e}tale Grothendieck topology on $Aff$; \\

(iii) $lis$ is the class of smooth morphisms in $Aff$. \\

The, we can verify that $(Aff,et, lis)$ is a geometric context.

\subsection*{Algebraic geometry context III}

Let $(Aff, nis, lis)$ be the triple where \\

(i) $Aff$ is the dual category of the category of commutative rings which are of finite presentation. \\

(ii) $nis$ is the Niesnevich Grothendieck topology on $Aff$; \\

(iii) $lis$ is the class of smooth morphisms in $Aff$. \\

The, we can verify that $(Aff,nis, lis)$ is a geometric context. 

\subsection*{Algebraic geometry context IV}

Let $(Aff, et, et)$ be the triple where \\

(i) $Aff$ is the dual category of the category of commutative rings which are of finite presentation. \\

(ii) $et$ is the \'{e}tale topology on $Aff$; \\

(iii) $et$ is the class of \'{e}tale morphisms in $Aff$. \\

The, we can verify that $(Aff,et, et)$ is a geometric context. \\

In general, we can proof without difficult the following \footnote{In order to proof (\ref{Geometric Site}), regard that open immersions are cartesian in $\mathcal{T}op$, and any open immersion in $\mathbf{C}$ is an element of $\mathbf{P}$, which implies (using (2) of (i)) that the open immersions are also cartesian in $\mathbf{C}$, and using these facts, we can verify (GC4), (GC5) and (GC6) (after verifying the first three conditions of (\ref{Geometric context}), which are straightforward).}

\begin{proposition}\label{Geometric Site}
  Let $(\mathbf{C}, J, \mathbf{P})$ be a triple where \\

  (i) $\mathbf{C}$ is a small (not necessarily full) subcategory of $\mathcal{T}op$ such that
  \begin{enumerate}
    \item $\mathbf{C}$ admits finite products
    \item If $U \hookrightarrow V$ is an open immersion and $V \in Ob(\mathbf{C})$, then $U \in Ob(\mathbf{C})$ 
    
\end{enumerate}    
  
   (ii)  $J$ be the Grothendieck topology generated by the open coverings in $\mathbf{C}$  \\
   
   (iii)  $\mathbf{P}$ be the class of local morphisms in $\mathbf{C}$, i.e., arrows $\varphi: U \rightarrow V$ in $\mathbf{C}$ satisfying the following condition: for each $x \in U$, there exist open subsets $U' \subseteq U$ and $V' \subseteq V$, with $x \in U'$ and $\varphi(x) \in V'$, such that $\varphi': U' \rightarrow V'$ is invertible in $\mathbf{C}$, where $\varphi'$ denotes the restriction of $\varphi$ to $U'$. \\
   
    Therefore, with the previous notations, we can verify that the triple $(\mathbf{C}, J,\mathbf{P})$ is a geometric context.

\end{proposition}

From the above proposition, one can verify that the topological, differential and analytic examples we have exposed, are actually geometric contexts according to the definition (\ref{Geometric context}).

\section{The Language of Elementary Schemes}

In this section, we develop the language of elementary schemes in a geometric context.

\begin{definition}
   Let $(C,J,\mathbf{P})$ be a geometric context and 
      $$\mathcal{U} =(p_{i}: F_{i} \longrightarrow F)_{i \in I}$$
       be a family of sheaf morphisms in $Sh(C,J)$. We say that $\mathcal{U}$ is a \emph{covering} if the canonical induced arrow
        $$ \coprod_{i \in I} F_{i} \longrightarrow F $$
is an epimorphism of sheaves. We say that $\mathcal{U}$ is a \emph{open covering} if it is a covering and each one of the morphisms $p_{i}$, for $i \in I$, are open immersions. Finally, we say that $\mathcal{U}$ is an \emph{open atlas} if it is an open covering and each one of the sheaves $F_{i}$, for $i \in I$, is representable by an object $U_{i}$ of $C$.

\end{definition}

Therefore, to give an \emph{open atlas} of a sheaf $X$ in a geometric context $(C,J,\mathbf{P})$ means to give a \emph{sheaf epimorphism} of the form
       $$ \coprod_{i \in I} h_{U_{i}} \longrightarrow X, $$
for which the canonical arrows $p_{j}: h_{U_{j}} \rightarrow X$, $j \in I$ (such that the diagrams
\[
\xymatrix{
   h_{U_{j}} \ar[d] \ar[dr]^{p_{j}}    &   \\
  \coprod_{i \in I} h_{U_{i}} \ar[r]_{p}  &  X
}
\]
commute), are all \emph{open immersions of sheaves}. We recall that to give a $J$-covering $(\rho_{i}:U_{i} \rightarrow U)_{i \in I}$ in $C$ is equivalent to give an epimorphism of sheaves of the form
     $$ \coprod_{i \in I} h_{U_{i}} \longrightarrow h_{U}. $$
We say that a $J$-covering as above is a $\mathbf{P}$-\emph{covering} (resp. $\mathbf{P}$-\emph{open covering}) when $\rho_{i}$ is a $\mathbf{P}$-morphism (resp. a $\mathbf{P}$-monomorphism) for each $i \in I$. Hence, it follows from (\ref{Open immersion equivalence}) that to give a $\mathbf{P}$-\emph{open covering} $(\rho_{i}: U_{i} \rightarrow U)_{i \in I}$ is equivalent to give an \emph{open atlas} of the form 
    $$ \coprod_{i \in I} h_{U_{i}} \rightarrow h_{U}. $$

\begin{definition}\label{S-morphism of sheaves}
  Let $(C,J,\mathbf{P})$ be a geometric context and $S$ be a set of arrows in $C$. A morphism $f: F \rightarrow G$ of $Sh(C,J)$ is called a $S$-morphism if for every arrow of the form $h_{U} \rightarrow G$, with $U \in Ob(C)$, there exists an \emph{open atlas} 
       $$ \coprod_{i \in I} h_{U_{i}} \longrightarrow F \times_{G} h_{U}, $$
such that the morphisms $U_{i} \rightarrow U$, induced from the canonical compositions
      $$ h_{U_{i}} \rightarrow \coprod_{i \in I} h_{U_{i}} \longrightarrow F \times_{G} h_{U} \longrightarrow h_{U_{i}} $$
are all $S$-morphisms in $C$.

\end{definition}

\begin{proposition}
   Let $(C,J,\mathbf{P})$ be a geometric context and $S$ be a class of arrows in $C$ satisfying the following conditions
\begin{enumerate}
 \item $S$ is stable by base change.
 \item $S$ is $J$-local: if $\varphi: U \rightarrow V$ is a morphism of $C$ and there exists a $J$-covering $(\rho_{i}: U_{i} \rightarrow U)_{i \in I}$ such that $\rho_{i} \in \mathbf{P}$ and $\varphi \circ \rho_{i} \in S$ for all $i \in I$, then $\varphi \in S$.
 \item $S$ is locally cartesian: if $\varphi: U \rightarrow V$ is a $S$-morphism then, there exists a $\mathbf{P}$-open covering $(\rho_{i}: U_{i} \rightarrow U)_{i \in I}$ such that $\varphi \circ \rho_{i} \in S$ and $\varphi \circ \rho_{i}$ is cartesian for every $i \in I$.
 
\end{enumerate}   
 Therefore, an arrow $\varphi: U \rightarrow V$ of $C$ is a $S$-morphism if, and only if, $h_{\varphi}: h_{U} \rightarrow h_{V}$ is a $S$-morphism of sheaves in the sense of (\ref{S-morphism of sheaves}).

\end{proposition}

\begin{proof}
 First, suppose that $h_{\varphi}: h_{U} \rightarrow h_{V}$ is a $S$-morphism of sheaves in the sense of (\ref{S-morphism of sheaves}). Then, there exists an open atlas 
               $$ \coprod_{i \in I} h_{U_{i}} \longrightarrow h_{U}, $$
such that each one of the arrows $U_{i} \rightarrow U$, induced by the evident compositions, are $S$-morphisms. Since the previous open atlas defines a $J$-covering $(\rho_{i}: U_{i} \rightarrow U)_{i \in I}$ such that $\rho_{i} \in \mathbf{P}$ and $\varphi \circ \rho_{i} \in S$ for every $i \in I$, then, by the $J$-locality of $S$, $\varphi \in S$. Now, suppose that $\varphi: U \rightarrow V$ is a $S$-morphism in $C$. Given an arrow $h_{W} \rightarrow h_{V}$, which corresponds, then, to an arrow $\psi: W \rightarrow V$ in $C$, we have to show that there is an open atlas of $h_{U} \times_{h_{V}} h_{W}$ satisfying the conditions of the definition (\ref{S-morphism of sheaves}). From the fact that $S$ is locally cartesian, there exists a $\mathbf{P}$-open covering $(\rho_{i}: U_{i} \rightarrow U)_{i \in I}$ such that $\varphi \circ \rho_{i} \in S$ and $\varphi \circ \rho_{i}$ is cartesian for all $i \in I$, , which implies in the existence of cartesian squares of the form
\[
\xymatrix{
    W_{i} \ar[r]^{\varphi_{i}'} \ar[d]_{\psi_{i}}     &    W \ar[d]^{\psi} \\
    U_{i}  \ar[r]_{\varphi_{i}}                       &    V,
}
\]
where $\varphi_{i}= \varphi \circ \rho_{i}$ and $\varphi_{i}' \in S$ (since $S$ is stable by base change). From the previous cartesian squares and from the universal property of fibred products, we can deduce easily a family of diagrams
\[
\xymatrix{
h_{W_{i}} \ar[r] \ar[d]  &  h_{U} \times_{h_{V}} h_{W} \ar[r] \ar[d]   &     h_{W} \ar[d] \\
h_{U_{i}} \ar[r]         &  h_{U} \ar[r]                               &     h_{V}
}
\]
in $Sh(C,J)$, where all the squares are cartesian (including the external one). Using the fact that epimorphisms in any topos are stable by base change, we can verify that the canonical arrow
       $$ p:\coprod_{i \in I} h_{W_{i}} \longrightarrow h_{U} \times_{h_{V}} h_{W} $$
is an epimorphism of sheaves. Moreover, $p$ is also an open atlas, because in the cartesian squares
\[
\xymatrix{
  h_{W_{i}} \ar[r] \ar[d]   &   h_{U} \times_{h_{V}} h_{W} \ar[d] \\
  h_{U_{i}}       \ar[r]    &   h_{U},
}
\]
the arrows $h_{U_{i}} \rightarrow h_{U}$ are all open immersions, and open immersions of sheaves are stable by base change. Hence, there exists an open atlas of $h_{U} \times_{h_{V}} h_{W}$ satisfying the conditions of (\ref{S-morphism of sheaves}), from where we conclude that $h_{\varphi}: h_{U} \rightarrow h_{V}$ is actually a $S$-morphism.

\end{proof}

\begin{corollary}
Let $(C,J,\mathbf{P})$ be a geometric context and $\varphi: U \rightarrow V$ an arrow of $C$. Then, $\varphi$ is a $\mathbf{P}$-morphism in $C$ if, and only if, $h_{\varphi}: h_{U} \rightarrow h_{V}$ is a $\mathbf{P}$-morphism of sheaves in the sense of (\ref{S-morphism of sheaves}).

\end{corollary}

\begin{corollary}\label{Open immersion = P-monomorphism}
  A morphism of sheaves $f: F \rightarrow G$ in a geometric context $(C,J,\mathbf{P})$ is an open immersion if, and only if, it is both a monomorphism and a $\mathbf{P}$-morphism of sheaves.

\end{corollary}

\emph{Notation and Terminology} - Given a geometric context $(C,J,\mathbf{P})$, we denote by $\widetilde{\mathbf{P}}$ (resp. $ouv$) the class of $\mathbf{P}$-morphisms (resp. open immersions) in the topos $Sh(C,J)$. Let $X$ be a sheaf over $(C,J)$. An \emph{atlas} of $X$ is a \emph{sheaf} epimorphism of the form
       $$ p: \coprod_{i \in I} h_{U_{i}} \longrightarrow X $$
such that $U_{i} \in Ob(C)$ for every $i \in I$, and the canonical induced arrows $h_{U_{i}} \rightarrow X$ are all $\mathbf{P}$-morphisms of sheaves in the sense of (\ref{S-morphism of sheaves}). \\

In the following, we define the notions of elementary schemes and geometric sheaves in a geometric context:

\begin{definition}\label{Elementary scheme, Geometric sheaf}
  Let $(C,J,\mathbf{P})$ be a geometric context. A sheaf $X$ over the site $(C,J)$ is called an \emph{elementary scheme} (resp. a geometric sheaf) in the geometric context $(C,J,\mathbf{P})$ if there exists at least one \emph{open atlas} (resp. \emph{atlas}) of the form
       $$ p: \coprod_{i \in I} h_{U_{i}} \longrightarrow X. $$

\end{definition}

\emph{Notation and Terminology} - If $(C,J,\mathbf{P})$ is a geometric context, we denote by $Sch(C,J,\mathbf{P})$ (resp. $Esg(C,J,\mathbf{P})$) the category of elementary schemes (resp. geometric sheaves) in $(C,J,\mathbf{P})$, which is the full subcategory of $Sh(C,J)$ formed by the elementary schemes (resp. geometric sheaves).
From the definition of elementary schemes, we have that any representable sheaf is an elementary scheme and any elementary scheme is a geometric sheaf. Moreover, any sheaf isomorphic to an elementary scheme (resp. a geometric sheaf) is also an elementary scheme (resp. a geometric sheaf). An elementary scheme which is representable by an object of $C$ will be called an \emph{elementary affine scheme}. Therefore, we have the following sequence of categorical immersions:
     $$ C \hookrightarrow Sch(C,J,\mathbf{P}) \hookrightarrow Esg(C,J,\mathbf{P}) \hookrightarrow Sh(C,J) \hookrightarrow \widehat{C}. $$
Given an open atlas (resp. an atlas) 
     $$ p: \coprod_{i \in I} h_{U_{i}} \longrightarrow X $$
of an elementary scheme (resp. of a geometric sheaf) $X$, we say that each one of the canonical arrows $p_{i}: h_{U_{i}} \rightarrow X$ is an \emph{open local chart} (resp. a \emph{local chart}) of $X$.

\begin{definition}\label{Schematic categories}
  A \emph{schematic category} $\mathscr{S}$ is a category equivalent to the category of elementary schemes in a geometric context, i.e., there exists a geometric context $(C,J,\mathbf{P})$ such that
                     $$\mathscr{S} \simeq Sch(C,J,\mathbf{P})  $$

\end{definition}

From the examples of geometric contexts we have presented in the first section, we can derive the following list of schematic categories:

\subsection*{Topological manifolds}       
  The category $\mathcal{M}$ of \emph{topological manifolds} is equivalent to the schematic category $Sch(\mathbf{C}_{0}, ouv, et)$. The category $Esg$ of \emph{geometric spaces} is the category of geometric sheaves in the context $(\mathbf{C}_{0}, ouv, et)$.

\subsection*{Differential manifolds}  
  The category $Diff$ of \emph{differential manifolds} is equivalent to the schematic category $Sch(\mathbf{C}_{\infty}, O_{\infty}, diff)$. The category $Es.diff$ of \emph{differential spaces} is the category of geometric sheaves in the context $(\mathbf{C}_{\infty}, O_{\infty}, diff)$.

\subsection*{Analytic manifolds}
  The category $An$ of (complex) \emph{analytic manifolds} is equivalent to the schematic category $Sch(\mathbf{C}_{h}, O_{h}, lis)$. The category $Es.an$ of (complex) \emph{analytic spaces} is the category of geometric sheaves in the context $(\mathbf{C}_{h}, O_{h}, lis)$.

\subsection*{Schemes}
  The category $Sch$ of \emph{schemes} in algebraic geometry is equivalent to the schematic category $Sch(Aff, et, lis)$. The category $Es.alg$ of \emph{algebraic spaces} is the category of geometric sheaves in the context $(Aff,et,lis)$.

\begin{remark}
  A schematic category $\mathscr{S}$ can be presented by more than one geometric context (just like a topos can be defined as a category of sheaves for different sites). For example, the category of schemes $Sch$ in algebraic geometry can be defined as the (same) category of elementary schemes in all the following geometric contexts:
 \begin{enumerate}
    \item $(Aff, fpqc, fppf)$
    \item $(Aff, et, lis)$
    \item $(Aff, nis, lis)$
    \item $(Aff, et, et)$
 
\end{enumerate}   
   In this sense, two geometric contexts $(C,J,\mathbf{P})$ and $(C',J',\mathbf{P}')$ are \emph{Morita-equivalent} if $Sch(C,J,\mathbf{P}) \simeq Sch(C',J',\mathbf{P}')$, similarly to the Morita-equivalence of sites, which currently occupies a crucial place in topos theory in the light of Caramello.  In order to have an intrinsic study of schematic categories, we need to be able to deduce a list of `Giraud axioms' for schematic categories, analogously the case of topos theory, where the Giraud axioms (the Lawvere-Tierney axioms, for an elementary topos), provides an intrinsic categorical characterization of a topos. We ignore if there exists such characterization, but we will see in the end of this section that the elementary schemes can be described as certain colimits of diagrams in the category $C$. 

\end{remark}

\begin{definition}\label{Schemeatic morphism}
 Let $(C,J,\mathbf{P})$ be a geometric context. A morphism $f: F \rightarrow G$ of sheaves over the site $(C,J)$ is called \emph{schematic} (resp. \emph{geometric}, \emph{affine}) if for every arrow of the form $h_{U} \rightarrow G$, with $U \in Ob(C)$, the sheaf $F \times_{G} h_{U}$ is an elementary scheme (resp. a geometric space, an affine scheme).

\end{definition}

\begin{lemma}\label{Schematic lemma}
   Let $(C,J,\mathbf{P})$ be a geometric context and $f: F \rightarrow X$ be a morphism of sheaves over the site $(C,J)$. If $X$ is an elementary scheme and there exists an open atlas 
       $$ p: \coprod_{i \in I} h_{U_{i}} \longrightarrow X $$
such that $F \times_{X} h_{U_{i}}$ is an elementary scheme for every $i \in I$, then $F$ is also an elementary scheme.

\end{lemma}

\begin{proof}
  Under the hypothesis of the lemma, we have the Cartesian square
\[
\xymatrix{
  \coprod_{i \in I} F \times_{X} h_{U_{i}} \ar[r] \ar[d]_{p'}   &  \coprod_{i \in I} h_{U_{i}} \ar[d]^{p} \\
   F \ar[r]_{f}                                                 &   X
}
\]
with $p$ being an epimorphism, which implies that $p'$ is also an epimorphism (since in any topos, epimorphisms are stable under base change). Moreover, it follows from the fact that open immersions are stable under base change, that the arrows $p'_{i}: F \times_{X} h_{U_{i}} \rightarrow F$, induced from $p'$, are all open immersions, because that is the case for the arrows $p_{i}: h_{U_{i}} \rightarrow X$ induced from $p$. Beyond that, there exists, for each $i \in I$, an open atlas
        $$ r_{i}: \coprod_{a \in A_{i}} h_{U_{i,a}} \longrightarrow F \times_{X} h_{U_{i}}, $$
such that the arrows $r_{i,a}: h_{U_{i,a}} \rightarrow F \times_{X} h_{U_{i}}$, induced by $r_{i}$, are all open immersions, which implies in the existence of a canonical arrow
    $$ r: \coprod_{i \in I} (\coprod_{a \in A_{i}} h_{U_{i,a}}) \longrightarrow \coprod_{i \in I} F \times_{X} h_{U_{i}}. $$
With the previous notations, we can verify that $r$ is also an epimorphism, and the composed arrows 
$$p''_{i,a}: h_{U_{i,a}} \xrightarrow{r_{i,a}} F \times_{X} h_{U_{i}} \xrightarrow{p'_{i}} F$$
are open immersions (since open immersions of sheaves are stable by composition). Then, defining the set
         $$ A =_{df} \coprod_{i \in I} A_{i} $$
we have the composed epimorphism
     $$p'': \coprod_{\alpha \in A} h_{U_{\alpha}} \xrightarrow{r}  \coprod_{i \in I} F \times_{X} h_{U_{i}} \xrightarrow{p'} F, $$
where the morphisms $p''_{\alpha}:h_{U_{\alpha}} \rightarrow F$, induced from $p''$, are all open immersions, for they coincide with the preceding composed morphisms $p''_{i,a}: h_{U_{i,a}} \rightarrow F \times_{X} h_{U_{i}} \rightarrow F$ (for $\alpha$ denoting the pair $(i,a)$ in $A$). Therefore, the sheaf $F$ admits an open atlas given by
     $$ p'': \coprod_{\alpha \in A} h_{U_{\alpha}} \longrightarrow F, $$
from where we conclude that $F$ is actually an elementary scheme. 

\end{proof}

\begin{corollary}
 Let $(C,J,\mathbf{P})$ be a geometric context and $f: F \rightarrow X$ be a morphism of sheaves over the site $(C,J)$. If $X$ is an elementary scheme  and $f$ is schematic, then $F$ is also an elementary scheme.

\end{corollary}

\begin{aforisma}\label{Schematic morphism digression}
  If $\mathscr{S}$ is a schematic category and $f:S' \rightarrow S$ is a schematic morphism between objects of $\mathscr{S}$, then, for every elementary scheme $p: X \rightarrow S$ over $S$, we can form an elementary scheme $p':X' \rightarrow S'$ over $S'$, where $X' = S' \times_{S} X$. In fact, let $(C,J,\mathbf{P})$ be a geometric context such that $\mathscr{S} \simeq Sch(C,J,\mathbf{P})$ and $\mathcal{E}$ be the topos $Sh(C,J)$. Considering the cartesian square
  \[
  \xymatrix{
    X' \ar[r]^{f'} \ar[d]_{p'}  &      X \ar[d]^{p} \\
    S' \ar[r]_{f}              &      S
  }
  \]
we can proof that $f'$ satisfies the hypothesis of the preceding corollary, because, for every morphism of the form $s:h_{U} \rightarrow X$, with $U \in Ob(C)$, we have the diagram
\[
\xymatrix{
  X'' \ar[r]^{f''} \ar[d]_{s'}     & h_{U} \ar[d]^{s} \\
   X' \ar[r]^{f'} \ar[d]_{p'}            &  X \ar[d]^{p} \\
   S' \ar[r]_{f}                         &  S
}
\]
in $\mathcal{E}$, where $X'' =  X' \times_{X} h_{U}$ and all the squares are cartesian (including the external one). Now, since $f$ is supposed to be schematic, $X''$ is an elementary scheme, from where we deduce that $f'$ is a schematic morphism from the sheaf $X'$ to the elementary scheme $X$, and conclude (in virtue of the previous corollary) that $X'$ is also an elementary scheme. In particular, if $f:S' \rightarrow S$ is a $\mathbf{P}$-morphism (resp. an open immersion), then the previous conclusion is also valid, for every $\mathbf{P}$-morphism (resp. open immersion) of sheaves is, by definition, schematic. 

\end{aforisma}

\begin{aforisma}
We remark that if $f: S' \rightarrow S$ is a schematic morphism, then it induces a canonical functor
     $$ f^{\ast}: \mathscr{S}_{/S} \longrightarrow \mathscr{S}_{/S'} $$
which admits a right adjoint
    $$ f_{\ast}: \mathscr{S}_{/S'} \longrightarrow \mathscr{S}_{/S}. $$
Moreover, if we denote by $\mathscr{S}m_{/S}$ the subcategory of $\mathscr{S}_{/S}$ formed by the objects $p: X \rightarrow S$ such that $p$ is a $\mathbf{P}$-morphism (inspired by the algebraic geometry situation), then, it also induces a functor
    $$ f^{\ast}: \mathscr{S}m_{/S'} \longrightarrow \mathscr{S}m_{/S}, $$
(since $\mathbf{P}$-morphisms are stable by base change) which has a right adjoint
    $$ f_{\ast}: \mathscr{S}m_{/S} \longrightarrow \mathscr{S}m_{/S'} $$
(since $\mathbf{P}$-morphisms are stable by base change).

\end{aforisma}

\begin{lemma}
  Let $(C,J,\mathbf{P})$ be a geometric context and $(F_{i})_{i \in I}$ be a multiple of sheaves over the site $(C,J)$. Define
         $$ F =_{df} \coprod_{i \in I} F_{i}. $$
Then, for every $i \in I$, the canonical inclusion arrow
       $$ s_{i}: F_{i} \longrightarrow F $$
is an open immersion of sheaves.

\end{lemma}

\begin{proof}
 Let $s:h_{U} \rightarrow F$ be a morphism of sheaves where $U \in Ob(C)$ and consider the canonical inclusion $s_{j}: F_{j} \rightarrow F$ for a fixed $j \in I$. By the definition of coproducts of sheaves, there exists a $J$-covering $(\varphi_{a}: U_{a} \rightarrow U)_{a \in A}$ such that, for every $a \in A$, there is a commutative square of the form:
  \[
  \xymatrix{
      h_{U_{a}} \ar[d]_{s_{ia}} \ar[r]^{h_{\varphi_{a}}}   &       h_{U} \ar[d]^{s} \\
      F_{i} \ar[r]_{s_{i}}      &       F,
  }
  \]
for some $i \in I$ and some morphism $s_{ia}: h_{U_{a}} \rightarrow F_{i}$. Hence, we can define the subset $A_{j}$ of $A$ composed only by the indices $a \in A$ for which the above condition is satisfied for $j$, i.e., the indices $a \in A$ for which there exists a commutative square of the form
 \[
  \xymatrix{
      h_{U_{a}} \ar[d]_{s_{ja}} \ar[r]^{h_{\varphi_{a}}}   &       h_{U} \ar[d]^{s} \\
      F_{j} \ar[r]_{s_{j}}      &       F.
  }
  \]
With the previous notations, we have a canonical arrow
 $$ p_{j}: \coprod_{i \in A_{j}} h_{U_{a}} \longrightarrow h_{U}, $$
and we can verify easily that the image of $p_{j}$ is isomorphic to the image of the evident induced monomorphism $F_{j} \times_{F} h_{U} \rightarrow h_{U}$ \footnote{We recall that the arrow $s_{j}: F_{j} \rightarrow F$ is a monomorphism and monomorphisms are stable by base change.}. Now, applying (GC5) of (\ref{Geometric context}), we can suppose, without any lost of generality, that each one of the arrows $\varphi_{a}: U_{a} \rightarrow U$ are $\mathbf{P}$-morphisms, which implies that $s_{j}: F_{j} \rightarrow F$ is actually a open immersions of sheaves.

\end{proof}

\begin{proposition}\label{Finite producrs and coproducts of schemes}
 Any schematic category admits finite products and arbitrary small coproducts.

\end{proposition}

\begin{proof}
  Let $\mathscr{S}$ be a schematic category and suppose that $(C,J,\mathbf{P})$ is a geometric context for $\mathscr{S}$, i.e., $\mathscr{S}$ is equivalent to the category of elementary schemes in $(C,J,\mathbf{P})$. Then, we just have to proof that $Sch(C,J,\mathbf{P})$ admits finite products and arbitrary small coproducts. We staring by proving that $Sch(C,J,\mathbf{P})$ admits finite products. Since the category $C$ admits finite products, it admits a terminal object $pt$, and since every local model is an elementary scheme, $pt$ is also a terminal object in the category $Sch(C,J,\mathbf{P})$. Now, if $X$ and $Y$ are two elementary schemes in $(C,J,\mathbf{P})$, consider an open atlas
      $$ p: \coprod_{i \in I} h_{U_{i}} \longrightarrow X $$
for $X$ and an open atlas
      $$ q: \coprod_{j \in J} h_{V_{j}} \longrightarrow Y $$
for $Y$. Then, we can form an open atlas
     $$ p \times q: \coprod_{(i,j) \in I \times J} h_{U_{i,j}} \longrightarrow X \times Y $$
where $U_{i,j} = U_{i} \times U_{j}$ in $C$, which proves that $\mathscr{S}$ is stable by finite products. To prove that $\mathscr{S}$ also admits arbitrary small coproducts is enough to verify that, if $(X_{i})_{i \in I}$ is a family of elementary schemes, then $\coprod_{i \in I} X_{i}$ is also an elementary scheme. Let
      $$ X =_{df} \coprod_{i \in I} X_{i}. $$
For each $i \in I$, there exists an open atlas 
      $$ p_{i}: \coprod_{a \in A_{i}} h_{U_{i,a}} \longrightarrow X_{i}, $$
and we have from the previous lemma that the canonical inclusions $s_{i}: X_{i} \rightarrow X$ are all open immersions, which implies that the composed arrows $p'_{i,a}: h_{U_{i,a}} \rightarrow X$, given by $p'_{i,a}= s_{i} \circ p_{i,a}$, are also all open immersions (because open immersions are stable by composition). Hence, with the notation
    $$ A = \coprod_{i \in I} A_{i} $$
we have for each $\alpha \in A$, an open immersion $p_{\alpha}: h_{U_{\alpha}} \rightarrow X$ \footnote{Regard that the elements $\alpha$ of $A$ are pairs $(i,a)$, where $i \in I$ and $a \in A_{i}$, and, by definition, $p_{\alpha} = p'_{i,a}$.}, from where we deduce the existence of an open atlas    
    $$ p: \coprod_{\alpha \in A} h_{U_{\alpha}} \longrightarrow X $$
and conclude that $\mathscr{S}$ is actually stable by arbitrary small coproducts.

\end{proof}

\begin{remark}
  In general, a schematic category is neither complete or co-complete, and may not even be finitely complete. For example, it's well known that the category of differential manifolds (which is a schematic category) does not admits arbitrary finite limits. The representability of limits and colimits of certain type in a schematic category depends on the category of local models, since these categories are locally defined (if we consider the category $Sch$ of schemes in algebraic geometry, then the finite limits are proven to be representable in $Sch$). Yet, since every schematic category $\mathscr{S}$ is, by definition, a full subcategory of a topos $\mathcal{E}$, we can always compute into the topos $\mathcal{E}$, the limits and colimits of diagrams in $\mathscr{S}$. The formalism of geometric context is very flexible in this sense, since it maintains the category $\mathscr{S}$ of geometric objects of interest, but inside a topos $\mathcal{E}$, which is a category with plenty completion and exact properties. 

\end{remark}

The next proposition is not difficult to proof, but very laboriously. 

\begin{proposition}\label{Fibred products of schemes}
  Let $(C,J,\mathbf{P})$ be a geometric context and $\mathscr{S}$ be the category of elementary schemes in $(C,J,\mathbf{P})$. If $C$ admits fibred products, then $\mathscr{S}$ also admits fibred products.

\end{proposition}

\emph{Sketch of the proof} - The strategy to proof the proposition is to apply (\ref{Schematic lemma}). Indeed, it follows from (\ref{Schematic lemma}) that it is enough to proof the statement of the proposition to the cartesian squares of the form
\[
\xymatrix{
  F \ar[r]^{f'} \ar[d]_{g'}   &   h_{U} \ar[d]^{g} \\
  h_{U'} \ar[r]_{f}       &   X
}
\]
in $Sh(C,J)$, with $U,U' \in Ob(C)$ and $X$ being elementary scheme. In this case, we can choose an open atlas
   $$ p: \coprod_{i \in I} h_{V_{i}} \longrightarrow X $$
of $X$, and, for every $i \in I$, we can form the Cartesian squares
\[
\xymatrix{
 h_{V_{i}} \times_{X} h_{U} \ar[r]^{p_{i}'} \ar[d]_{g_{i}}  &   h_{U} \ar[d]^{g} \\
 h_{V_{i}}          \ar[r]_{p_{i}}                 &   X
}
\]
and
\[
\xymatrix{
 h_{V_{i}} \times_{X} h_{U'} \ar[r]^{p_{i}''} \ar[d]_{f_{i}}  &   h_{U'} \ar[d]^{f} \\
 h_{V_{i}}          \ar[r]_{p_{i}}                 &   X
}
\]
 where $p_{i}'$ and $p_{i}''$ are open immersions (since open immersions are stable by base change). Then, $h_{V_{i}} \times_{X} h_{U}$ and $h_{V_{i}} \times_{X} h_{U'}$ are elementary schemes, and defining respectively open atlases for $h_{V_{i}} \times_{X} h_{U}$ and $h_{V_{i}} \times_{X} h_{U'}$, we derive the existence of open atlases $(h_{U'_{i}} \rightarrow h_{U})_{i \in I}$ and $(h_{U_{i}} \rightarrow h_{U})_{i \in I}$ \footnote{Using the stability by base change and the stability by compositions of epimorphisms, and also the stability by compositions of open immersions.}. With the previous notations, we can now reply an argument analogous to the one exposed in the page 23 of [10] to conclude the proof.

\begin{corollary}
 Under the hypothesis of the proposition (\ref{Fibred products of schemes}), every morphism between elementary schemes is schematic according to (\ref{Schemeatic morphism}).

\end{corollary}

The next theorem characterises the elementary schemes as a certain type of colimit in the category of sheaves in a geometric context, and is inspired from [10]:

\begin{theorem}\label{Characterization of elementary schemes}
 (To\"{e}n-Vaqui\'{e}) Let $(C,J,\mathbf{P})$ be a geometric context and $F$ be a sheaf over the site $(C,J)$. Equivalent conditions:
  \begin{enumerate}
    \item $F$ is an elementary scheme.
    \item There exists an equivalence relation $(R,X, r_{1},r_{2})$ in $Sh(C,J)$ \footnote{A brief definition of equivalence relation in a topos, as well it's main exactness properties, are exposed in the preliminary chapter of [9].} such that the following conditions are verified:
    \begin{enumerate}
      \item $X \cong \coprod_{i \in I} h_{U_{i}}$, where $U_{i} \in Ob(C)$ for every $i \in I$;
      \item For every pair $(i,j)$ of indices if $I$, there exists a sunfunctor $R_{ij} \hookrightarrow X_{i} \times X_{j}$ where the squares
      \[
      \xymatrix{
         R_{ij} \ar[rr]^{(r_{i},r_{j})} \ar[d]    &   &   h_{U_{i}} \times h_{U_{j}} \ar[d] \\
         R \ar[rr]_{(r_{1},r_{2})}                &   &  X \times X
      }
      \]
 are cartesian and the arrows $R_{ij} \rightarrow h_{U_{i}}$ are open immersions;  
      \item For every $i \in I$, the inclusion $R_{ii} \hookrightarrow h_{U_{i}} \times h_{U_{i}}$ coincides with the diagonal morphism;
      \item $F$ is a quotient $X/R$ of the the equivalence relation $(R,X,r_{1},r_{2})$.
      
    \end{enumerate}
    
  \end{enumerate}

\end{theorem}

\begin{proof}
  (1) $\implies$ (2): Suppose that $F$ is an elementary scheme. Then, there is at least one open atlas
          $$ p: \coprod_{i \in I} h_{U_{i}} \longrightarrow F $$
of $F$. Define
           $$ X =_{df} \coprod_{i \in I} h_{U_{i}} $$
and, for each pair $(i,j)$ of indices in $I$, define
           $$ R_{ij} =_{df} h_{U_{i}} \times_{F} h_{U_{j}}. $$
Then, define 
          $$ R =_{df} \coprod_{(i,j) \in I^{2}} R_{ij}. $$
With the previous notations, the reader can verify easily that the conditions (a), (b), (c) and (d) are satisfied \footnote{Note that the conditions (a), (c)  and (d) follow from the Giraud's axioms of a topos, while the condition (b) is a consequence of the stability properties of open immersions.}.  \\

(2) $\implies$ (1): This implication is essentially proved in the page 24 of [10].

\end{proof}

\begin{remark}
  The equivalence relation in the theorem (\ref{Characterization of elementary schemes}) is defined only in terms of small coproducts and finite products, which are always representable in any schematic category (from (\ref{Finite producrs and coproducts of schemes})). In fact, the theorem (\ref{Characterization of elementary schemes}) provides a prescription of how the objects in any schematic category are derived from the affine elementary schemes in terms of coproducts, finite products, and open immersions \footnote{Note that the sheaves $R_{ij}$ in the theorem (\ref{Characterization of elementary schemes}) are elementary schemes, since the structural arrows $R_{ij} \rightarrow h_{U_{i}}$ are open immersions, and from (\ref{Schematic morphism digression}), open immersions are schematic}, being a first step for a possible intrinsic (axiomatic) characterisation of schematic categories.

\end{remark}

\newpage

\begin{large}
\textbf{References} \\

\end{large} 

[1] M. Artin, A. Grothendieck, J.-L. Verdier - \emph{Th\'{e}orie des topos et cohomologie \'{e}tale des sch\'{e}mas} (SGA4), Lect. Notes in Math., vol. 269, 270, 305, Springer-Verlag, 1972-73 \\

[2]. A. Grothendieck and J. Dieudonn\'{e}, \emph{\'{E}lements de g\'{e}om\'{e}trie alg\'{e}brique}, Publ. Math. IHES 4(1960) 8(1961) 11(1961) 17(1963) 20(1964) 24(1965) 28(1966) 32(1967) \\

[3]. Jean-Pierre Serre, \emph{Faisceaux alg\'{e}briques coh\'{e}rents}, Ann.of Math. (2) 61, (1955). 192-278. \\

[4]. S\'{e}minaire CHEVALLEY, 1956-1958: \emph{Classification des groupes de Lie alg\'{e}briques}, Fssc. 1 et 2 - Paris, S\'{e}cr\'{e}tariat math\'{e}matique, 1958 \\

[5]. M. Demazure, P. Gabriel. \emph{Groupes Alg\'{e}briques. Tome I: G\'{e}om\'{e}trie Alg\'{e}brique, g\'{e}n\'{e}ralit\'{e}s, groups commutatifs}. Masson, Amsterdam: North Holland, 1970.\\

[6]. M. Demazure, \emph{Lectures on p-divisible groups}. Lectures notes in Mathematics, 302, Berlin: Springer-Verlag, 1972, 1986 \\

[7]. M. Kashiwara, P. Schapira: \emph{Categories and sheaves}, Grundlehren der Mathematischen Wissenschaften, \textbf{332} Springer, 2006. \\

[8]. H. Schubert, \emph{Categories}, Springer-Verlag, Berlin-Heidelberg-New-York, 1972. \\ 

[9]. Johnstone, P. T., \emph{Topos theory}, Academic Press (1977). \\

[10]. B. To\"{e}n and M. Vaqui\'{e}, \emph{Au-dessous de spec($\mathbb{Z}$)},. J. K-Theory 3 (2009) 437-500.\\

[11]. B. To\"{e}n and M. Vaqui\'{e}, \emph{Alg\'{e}brisation des vari\'{e}t\'{e}s analytiques complexes et cat\'{e}gories d\'{e}riv\'{e}es}, Math. Ann. 342 (2008) 789-831. \\

[12]. Vaquié, M. (n.d.). \emph{Sheaves and Functors of Points}. In (pp. 407-461). doi:10.1017/9781108854429.011 \\

[13]. P. Feit, \emph{Axiomatization of passage from local structure to global object}, Mem. Amer. Math. Soc. \textbf{101} (1993). \\

[14]. Low, Z. L. (2016). \emph{Categories of spaces built from local models} (Doctoral thesis). https://doi.org/10.17863/CAM.384) \\

[15]. B. To\"{e}n, G. Vezzosi, \emph{Homotopical Algebraic Geometry II: Geometric Stacks and applications}, Mem. Amer. Math. Soc. \textbf{193} (2) (2008). \\

[16]. Jean Giraud, \emph{Cohomologie non ab\'{e}lienne}, Springer (1971)

\end{document}